\documentclass[12pt,reqno]{amsart}
\usepackage[top=2cm,bottom=2cm,right=2.5cm,left=2.5cm]{geometry}
\usepackage{amssymb}
\usepackage{amsmath, amsthm, amscd, amsfonts, amssymb, graphicx, color}
\usepackage[bookmarksnumbered, colorlinks, plainpages]{hyperref}
\hypersetup{colorlinks=true,linkcolor=red, anchorcolor=green, citecolor=cyan, urlcolor=red, filecolor=magenta, pdftoolbar=true}
 
\usepackage{hyperref}

\newtheorem{theorem}{Theorem}[section]
\newtheorem{lemma}[theorem]{Lemma}
\newtheorem{proposition}[theorem]{Proposition}

\theoremstyle{definition}
\newtheorem{definition}[theorem]{Definition}
\newtheorem{example}[theorem]{Example}

\theoremstyle{remark}

\numberwithin{equation}{section}

\begin{document}
	
	\setcounter{page}{1}
	
	\title[Bessel sequences in Hilbert $C^{\ast}$-modules]{Bessel sequences in Hilbert $C^{\ast}$-modules}

	\author[A. Karara, K. Mabrouk]{Abdelilah Karara$^{*}$ and Khadija Mabrouk}
	
	\address{Department of Mathematics Faculty of Sciences, University of Ibn Tofail, Kenitra, Morocco}
	\email{\textcolor[rgb]{0.00,0.00,0.84}{abdelilah.karara@uit.ac.ma; khadija.mabrouk@uit.ac.ma}}

	
	
	\date{
		\newline \indent $^{*}$Corresponding author}

	\begin{abstract} In this peaper we stady certain Bessel sequences $\left\{f_k\right\}_{k=1}^{\infty}$ in Hilbert C*- modules $\mathcal{H}$ for which operator $S$ defined by \ref{eq2} is of the form $\mathcal{T}+\xi I$, for some real number $\xi$ and a adjointable linear operator $\mathcal{T}$. Additionally, we investigate frames known as compact-tight frames, which have frame operators that are compact perturbations of constant multiples of the identity. As a conclusion, we provide a theory regarding the weaving of specific compact-tight frames.
	\end{abstract}
	\maketitle
	\section{Introduction }
	
The notion of frame is a recent active mathematical research topic, signal processing, computer science, etc. Frames for Hilbert spaces were first introduced in \textbf{1952} by Duffin and Schaefer \cite{Duf} for study of nonharmonic Fourier series. Daubechies, Grossmann, and Meyer \cite{Dau} revived and developed them in \textbf{1986}, and popularized from then on.

 Many mathematicians have recently generalized frame theory from Hilbert spaces to Hilbert $C^*$-modules. For find details of frames in Hilbert $C^*$-modules we refer to \cite{Ghiati, Kho2,Mass,  Pas, r1, ROSS, ROS, Sah}. The purpose of this paper is to investigate.

Throughout this paper, Let $\mathcal{H}$ be a countably generated Hilbert $\mathcal{A}$-module and   $\{\mathcal{K}_k\}_{k \in \mathbb{Z}}$ be the collection of Hilbert $\mathcal{A}$-module, we also reserve the notation $\operatorname{End}_\mathcal{A}^*\left(\mathcal{H}, \mathcal{K}_k\right)$ for the collection of all adjointable $\mathcal{A}$-linear maps from $\mathcal{H}$ to $\mathcal{K}_{k}$ and $\operatorname{End}_\mathcal{A}^*(\mathcal{H}, \mathcal{H})$ is denoted by $\operatorname{End}_\mathcal{A}^*(\mathcal{H})$.  We will use $\mathcal{N}(\mathcal{T})$  for the kernel of $\mathcal{T}$.
$$ \ell^2= \bigg \{ \{f_k \}_{k \in \mathbb{N}} \,:  \,\,\,\,\,  \|\sum_{k \in \mathbb{N}} \langle f_k ,f_k \rangle \|< \infty \bigg \}$$ that is, $\mathcal{T}( \{f_k \}_{k \in \mathbb{N}})= \{f_{k +1}\}_{k \in \mathbb{N}}.$
Let $f=\left\{f_k\right\}_{k \in \mathbb{N}}$ and $g=\left\{g_k\right\}_{k \in \mathbb{N}}$,  the inner product is defined by $\langle f, g\rangle=\sum_{k \in \mathbb{N}}\left\langle f_k, g_k\right\rangle$, clearly $\ell^2$ is a Hilbert $\mathcal{A}$-module.

In the following we briefly recall the definitions and basic properties of Hilbert $C^{\ast}$-modules.

\begin{definition}\cite{Kap}.
	Let $ \mathcal{A} $ be a unital $C^{\ast}$-algebra and $\mathcal{H}$ be a left $ \mathcal{A} $-module, such that the linear structures of $\mathcal{A}$ and $\mathcal{H}$ are compatible. $\mathcal{H}$ is a pre-Hilbert $\mathcal{A}$-module if $\mathcal{H}$ is equipped with an $\mathcal{A}$-valued inner product $\langle.,.\rangle :\mathcal{H}\times \mathcal{H}\rightarrow\mathcal{A}$, such that is sesquilinear, positive definite and respects the module action. In the other words,  
	\begin{itemize}
		\item [(i)] $ \langle x,x\rangle\geq0 $ for all $ x\in \mathcal{H}$ and $ \langle x,x\rangle=0$ if and only if $x=0$.
		\item [(ii)] $\langle ax+y,z\rangle=a\langle x,z\rangle+\langle y,z\rangle$ for all $a\in\mathcal{A}$ and $x,y,z\in \mathcal{H}$.
		\item[(iii)] $ \langle x,y\rangle=\langle y,x\rangle^{\ast} $ for all $x,y\in \mathcal{H}$.
	\end{itemize}	 

For $x\in \mathcal{H}$, we define $||x||=||\langle x,x\rangle||^{\frac{1}{2}}$. If $\mathcal{H}$ is complete with $||.||$  is a norm on $\mathcal{H}$, it is called a Hilbert $\mathcal{A}$-module or a Hilbert $C^{\ast}$-module over $\mathcal{A}$. For every $a$ in $C^{\ast}$-algebra $\mathcal{A}$, we have $||a||=(a^{\ast}a)^{\frac{1}{2}}.$ 

\end{definition}

\begin{lemma}\cite{Pas}.
	Let $\mathcal{H}$ be Hilbert $\mathcal{A}$-module. If $\mathcal{T}\in End_{\mathcal{A}}^{\ast}(\mathcal{H})$, then $$\langle \mathcal{T}x,\mathcal{T}x\rangle_{\mathcal{A}}\leq\|\mathcal{T}\|^{2}\langle x,x\rangle_{\mathcal{A}}, \forall x\in\mathcal{H}.$$
\end{lemma}
\begin{definition}\cite{Kap} A family $ \{f_{k}\}_{k\in \Theta} $ of elements of $ \mathcal{H} $, if there exist two positive constants $ A $ and $ B $ such that for all $ f\in \mathcal{H} $, 
	\begin{equation}\label{eq1}
		A\langle f,f\rangle_{\mathcal{A}}\leq \sum_{k\in \Theta}  \langle f,f_{k}\rangle_{\mathcal{A}} \langle f_{k},f\rangle_{\mathcal{A}} \leq B \langle f,f\rangle_{\mathcal{A}}.
	\end{equation} 
	The numbers $ A $ and $ B $ are called lower and upper  bounds of the frame, respectively. If $ A = B = \alpha $, the frame is $ \alpha- $tight. If $ A= B=1 $, it is called a normalized tight frame or a Parseval frame.
\end{definition}	
	Let $\mathfrak{F}=\left\{f_k\right\}_{k=1}^{\infty}$ be a frame for $\mathcal{H}$. The operator 
$$
\mathcal{T}_\mathfrak{F}: \ell^2 \rightarrow \mathcal{H}, \quad \mathcal{T}_{\mathfrak{F}}\left(\left\{\xi_k\right\}_{k \in I}\right)=\sum_{k=1}^{\infty} \xi_k f_k,
$$
is called the synthesis operator of $\mathfrak{F}$, and
$$
\mathcal{T}_\mathfrak{F}^*: \mathcal{H} \rightarrow \ell^2, \quad \mathcal{T}_\mathfrak{F}^* f=\left\{\left\langle f, f_k\right\rangle\right\}_{k \in I}
$$
is called the analysis operator of $\mathfrak{F}$. The frame operator $S: \mathcal{H} \rightarrow \mathcal{H}$, defined by
\begin{equation} \label{eq2}
S f=\mathcal{T}_\mathfrak{F} \mathcal{T}_\mathfrak{F}^* f=\sum_{k=1}^{\infty}\left\langle f, f_k\right\rangle f_k,
\end{equation}

is a positive, self-adjoint, and invertible operator. 

\begin{definition}
Let $\mathcal{H}$ be a finitely or countably generated Hilbert $\mathcal{A}$-module $\mathcal{H}$ over a unital $C^*$-algebra $\mathcal{A}$, and let $\mathfrak{F}=\left\{f_k\right\}_{k=1}^{\infty}$ and $\mathfrak{G}=\left\{g_k\right\}_{k=1}^{\infty}$ be Bessel sequences in $\mathcal{H}$. If for $\mu>0$,
$$
\left\|\mathcal{T}_\mathfrak{F}-\mathcal{T}_\mathfrak{G}\right\| \leq \mu .
$$  Then $\mathfrak{G}$ is called a $\mu$-perturbation of $\mathfrak{F}$ 
\end{definition}

\begin{definition}\cite{Zah}
For a given natural number $m$. A finite family of frames $\left\{\left\{x_{i j}\right\}_{j \in \mathbb{J}}, i \in \lbrace 1,\ldots , m \rbrace \right\}$ in a Hilbert $\mathcal{A}$-module $\mathcal{H}$ is said to be woven if there are universal constants $C$ and $D$ such that, for every partition $\left\{\sigma_1, \ldots, \sigma_m\right\}$ of $\mathbb{J}$, the family $\left\{x_{i j}\right\}_{i=1, j \in \sigma_i}^m$ is a frame for $\mathcal{H}$ with lower and upper frame bounds $C$ and $D$, respectively. In this case, we usually call $\left\{\left\{x_{i j}\right\}_{j \in \mathbb{J}}, i \in \lbrace 1,\ldots , m \rbrace \right\}$ woven with universal bounds $(C, D)$. Each family $\left\{x_{i j}\right\}_{i=1, j \in \sigma_i}^m$ is called a weaving.
\end{definition}

\section{From Bessel sequences to frames}
\begin{proposition}\label{P31}
 Let $\left\{f_k\right\}_{k=1}^{\infty}$ be a Bessel sequence in $\mathcal{H}$ with bound $\eta$ and $S$ be the frame operator. if we can write $S$ in the following form $S=\mathcal{T}+\xi I$ for some real number $\xi$ and a linear operator $\mathcal{T}$, then we have:
 \begin{enumerate}
 \item $\left\{f_k\right\}_{k=1}^{\infty}$ is a frame. If $\mathcal{T}$ is a positive operator and $\xi>0$,
 \item $\mathcal{T}$ is bounded and self-adjoint.

 \item Let $A>0$, the operator $\mathcal{T}$ is a positive whenever $\xi \leq A$. If $\left\{f_k\right\}_{k=1}^{\infty}$ is a frame with lower bound $A$
 \end{enumerate}
\end{proposition}  
\begin{proof}

(1) For every $f \in \mathcal{H}$, 
$$\sum_{k=1}^{\infty}\left\langle f, f_k\right\rangle \left\langle f_k , f\right\rangle=\langle S f, f\rangle= \langle \xi f, f\rangle+\langle T f, f\rangle  $$
Since $\mathcal{T}$ is a positive operator then,
$$
\xi\left\langle f, f\right\rangle\leq\sum_{k=1}^{\infty}\left\langle f, f_k\right\rangle \left\langle f_k , f\right\rangle \leq\eta \left\langle f, f\right\rangle 
$$
which implies the sequence $\left\{f_k\right\}_{k=1}^{\infty}$ is a frame for $\mathcal{H}$.

(2) for $f\in \mathcal{H}$ we have, $$\sum_{k=1}^{\infty}\left\langle f, f_k\right\rangle \left\langle f_k , f\right\rangle\leq \left\langle f , f\right\rangle.$$ Thus $S$ is a bounded operator. Hence, $S-\xi I$ is also bounded. Since $\xi$ is a real number the operator $\mathcal{T}$ satisfies the following equalities $$\mathcal{T}=\mathcal{T}^{*}=S-\xi I$$ is obvious.

(3) For every $f \in \mathcal{H}$,
$$
\begin{aligned}
\langle \mathcal{T} f, f\rangle & =\langle S f, f\rangle-\langle \xi f, f\rangle \\
& \geq A\left\langle f, f\right\rangle-\xi\left\langle f, f\right\rangle
\\
& \geq(A-\xi)\left\langle f, f_k\right\rangle
\end{aligned}
$$
Morever, the hypothesis $\xi \leq A$ shows that for every $f \in \mathcal{H}$, $$\langle T f, f\rangle \geq 0.$$ 
\end{proof}

In the following theorem we admit the following hypothesis:

\begin{equation}\label{eq31}
\|f\|\|g\| \leq \sqrt{1+\eta^2}\|\langle f, g\rangle\| \Longleftrightarrow\|\alpha f-g\| \leq \sqrt{\frac{\eta^{2}}{1+\eta^2}}\|g\|, \quad f,g\in\mathcal{H}.
\end{equation}
Where $\alpha\in\mathbb{R}$ and $\eta\geq0$. 
\begin{theorem}
 Let $\mathcal{T}\in\operatorname{End}_\mathcal{A}^*(\mathcal{H})$ and $\left\{f_k\right\}_{k=1}^{\infty}\subset\mathcal{H}$ such that ,
$$
\sum_{k=1}^{\infty}\left\langle f, f_k\right\rangle f_k=\mathcal{T} f+\xi f,\quad \forall f \in \mathcal{H}
$$
where $\xi$ is a real scalar.
\begin{enumerate}
\item The famille $\left\{f_k\right\}_{k=1}^{\infty}$ is a Bessel sequence.
\item Suppose that  $\alpha$ is a real number and $\eta \geq 0$ are such that
\begin{equation}\label{eh}
\|\alpha f-\mathcal{T} f\| \leq \sqrt{\frac{\eta^{2}}{1+\eta^2}}\|\mathcal{T} f\|,\quad \forall f \in \mathcal{H}.
\end{equation}
 If $\mathcal{T}$ is self-adjoint and bounded from below by $\rho$ with
$$
\sqrt{\frac{\rho^{2}}{1+\eta^2}}-|\xi|>0
$$
then $\left\{f_k\right\}_{k=1}^{\infty}$ is a frame.
\item With (\ref{eh}), assume that $\left\{g_k\right\}_{k=1}^{\infty}$ is a $\mu$-perturbation of $\left\{f_k\right\}_{k=1}^{\infty}$. If $\mu<\sqrt{\sqrt{\frac{\rho^{2}}{1+\eta^2}}-|\xi|}$, then $\left\{g_k\right\}_{k=1}^{\infty}$ is a frame for $\mathcal{H}$ with bounds
$$
\left(\left(\sqrt{\frac{\rho^{2}}{1+\eta^2}}-|\xi|\right)^{\frac{1}{2}}-\mu\right)^2 \text { and }(\mu+(\|\mathcal{T}\|+|\xi|)^{\frac{1}{2}})^2 \text {. }
$$
\end{enumerate}
\end{theorem}
\begin{proof}
(1) Using Cauchy-Schwarz inequality we get
$$
\begin{aligned}
\Vert\sum_{k=1}^{\infty}\left\langle f, f_k\right\rangle \left\langle f_k , f\right\rangle\Vert & =\Vert\langle \mathcal{T} f, f\rangle+\langle \xi f, f\rangle\Vert \\
& \leq\Vert\langle \mathcal{T} f, f\rangle\Vert+ |\xi|\|f\|^2\\
& \leq\|\langle \mathcal{T} f, \mathcal{T} f\rangle\|^\frac{1}{2} \Vert\langle f,  f\rangle\Vert^{\frac{1}{2}}+|\xi|\|f\|^2\\
& \leq\|\mathcal{T}\| \Vert\langle f,  f\rangle\Vert +|\xi|\|f\|^2\\
& =(\|\mathcal{T}\|+|\xi|)\Vert\langle f,  f\rangle\Vert.
\end{aligned}
$$
(2) In view of \ref{eq31} and the assumption that $\mathcal{T}\in\operatorname{End}_\mathcal{A}^*(\mathcal{H})$ is self-adjoint we can write
$$
\begin{aligned}
\Vert\sum_{k=1}^{\infty}\left\langle f, f_k\right\rangle \left\langle f_k , f\right\rangle\Vert & =\Vert\langle \xi f, f\rangle+\langle \mathcal{T} f, f\rangle\Vert \\
& \geq\Vert\langle \mathcal{T} f, f\rangle\Vert-|\xi|\Vert\langle f, f\rangle \Vert \\
& =\Vert\langle f, \mathcal{T} f\rangle\Vert-|\xi|\Vert\langle f, f\rangle \Vert\\
& \geq \frac{1}{\sqrt{1+\eta^2}}\|\mathcal{T} f\|\|f\|-|\xi|\Vert\langle f, f\rangle \Vert,
\end{aligned}
$$
Since the operator $\mathcal{T}$ is bounded from below by $\rho$, we get
$$
\begin{aligned}
\Vert\sum_{k=1}^{\infty}\left\langle f, f_k\right\rangle \left\langle f_k , f\right\rangle\Vert 
& \geq \sqrt{\frac{\rho^{2}}{1+\eta^2}}\Vert\langle f, f\rangle \Vert-|\xi|\Vert\langle f, f\rangle \Vert \\
& =\left(\sqrt{\frac{\rho^{2}}{1+\eta^2}}-|\xi|\right)\Vert\langle f, f\rangle \Vert .
\end{aligned}
$$
(3) Since $\left\{f_k\right\}_{k-1}^{\infty}$ is a frame with bounds $\|\mathcal{T}\|+|\xi|$ and $\sqrt{\frac{\rho^{2}}{1+\eta^2}}-|\xi|$, for every $f \in \mathcal{H}$,
$$
\left(\sqrt{\frac{\rho^{2}}{1+\eta^2}}-|\xi|\right)^{\frac{1}{2}}\|f\| \leq\left\|U_{\mathcal{j}}^* f\right\| \leq(|\xi|+\| \mathcal{T} \|)^{\frac{1}{2}}\|f\| .
$$
Therefore,
$$
\left\|\mathcal{T}_\mathfrak{G}^*(f)\right\| \geq\left\|\mathcal{T}_\mathfrak{F}^*(f)\right\|-\left\|\left(\mathcal{T}_\mathfrak{F}(f)-\mathcal{T}_\mathfrak{G}(f)\right)^*\right\| \geq\left(\left(\sqrt{\frac{\rho^{2}}{1+\eta^2}}-|\xi|\right)^{\frac{1}{2}}-\mu\right)\|f\| .
$$
Furthermore,
$$
\left\|U_{\dot{G}}^*(f)\right\| \leq\left\|\left(\mathcal{T}_\mathfrak{F}(f)-\mathcal{T}_\mathfrak{G}(f)\right)^*\right\|+\left\|\mathcal{T}_\mathfrak{F}^*(f)\right\| \leq(\mu+(\|\mathcal{T}\|+|\xi|)^{\frac{1}{2}})\|f\| .
$$
Hence,
$$
\left( \sqrt{\frac{\rho^{2}}{1+\eta^2}}-|\xi|-\mu\right)^2\langle f,f\rangle \leq \sum_{i=1}^{\infty}\left\langle f, g_k\right\rangle \left\langle  g_k,f\right\rangle \leq(\mu+(\|\mathcal{T}\|+|\xi|)^{\frac{1}{2}})^2\langle f,f\rangle .
$$
\end{proof}

\section{Compact and finite-rank-tight frames}
We assume in this section that the compact self-adjoint operator $\mathcal{T}$ is written in the following form

\begin{equation}\label{eq41}
\mathcal{T}=\sum_{k=1}^{\infty} \alpha_k\left\langle\cdot, e_k\right\rangle e_k.
\end{equation}

Where $\left\{e_k\right\}_{k=1}^{\infty}$ is an orthonormal basis for $\mathcal{H}$ formed of the eigenvectors of $\mathcal{T}$. This fact is used to give conditions in the following theorem that enable us to assume a Bessel sequence is a frame.
\begin{theorem}
Let  $\mathcal{T}$ be a compact self-adjoint operator and  $\left\{f_k\right\}_k^{\infty}$ be a Bessel sequence such that for $\xi>0$, $S=\mathcal{T}+\xi I$. Assume that $\mathcal{T}$ has the forme of (\ref{eq41}) such that
$$
\alpha:=\inf _k \alpha_k+\xi>0 .
$$
Then, $\left\{f_k\right\}_{k=1}^{\infty}$ is a frame.
\end{theorem}
\begin{proof}
For every $f \in \mathcal{H}$, we have

$$
\begin{aligned}
Sf&=\mathcal{T}f+\xi f\\
 & =\sum_{k=1}^{\infty} \alpha_k\left\langle f, e_k\right\rangle e_k +\xi f \\
& =\sum_{k=1}^{\infty} \alpha_k\left\langle f, e_k\right\rangle e_k +\xi \sum_{k=1}^{\infty}\left\langle f, e_k\right\rangle e_k\\
& =\sum_{k=1}^{\infty}\left(\alpha_k+\xi\right)\left\langle f, e_k\right\rangle e_k .
\end{aligned}
$$
Consequently,
$$
\begin{aligned}
\sum_{k=1}^{\infty}\left\langle f, f_k\right\rangle \left\langle f_k, f\right\rangle & =\langle S f, f\rangle \\
& =\sum_{k=1}^{\infty}\left(\alpha_k+\xi\right)\left\langle f, e_k\right\rangle \left\langle e_k,f \right\rangle\\
& \geq \sum_{k=1}^{\infty}\left(\inf _k\alpha_k+\xi\right)\left\langle f, e_k\right\rangle \left\langle e_k,f \right\rangle\\
& \geq \alpha\left\langle f,f \right\rangle .
\end{aligned}
$$
Where $\alpha=\xi+\inf_k \alpha_k>0 .$
\end{proof}
\begin{definition}\label{D42}
 We say that a frame is compact-tight. If its frame operator $S$ satisfies the following condition 
 \begin{equation}\label{ED1}
 S=\mathcal{K}+\xi I
 \end{equation}
 where $\mathcal{K}$ being a compact operator.
\end{definition}
The previous definition remains true for finite-rank.

Let $\left\{f_k\right\}_{k=1}^{\infty}$ be a compact-tight frame such that $$S=\mathcal{K}_1+\xi_1 I=\mathcal{K}_2+\xi_2 I,$$ where $\mathcal{K}_1$ and $\mathcal{K}_2$ being compact operators. Then, 
\begin{equation}\label{eq42}
\mathcal{K}_2-\mathcal{K}_1=\left(\xi_1-\xi_2\right) I.
\end{equation}

Thus $$\xi_1=\xi_2.$$  The operator to the right of is indeed a compact operator, but the operator to the left cannot be compact unless $$\xi_1-\xi_2=0.$$ Consequently $$\mathcal{K}_1=\mathcal{K}_2.$$
\begin{proposition}\label{p34}
The representation of the frame operator $S$ as a compact perturbation of the form (\ref{ED1}) is unique. If $\left\{f_k\right\}_{k=1}^{\infty}$ is compact-tight frame.

\end{proposition}  
 We can formulate the proposition \ref{p34} for tight finite rank frames like this

we say that $\left\{f_k\right\}_{k=1}^{\infty}$ is a $(\xi, \mathcal{K})$-compact-tight frame. If $\left\{f_k\right\}_{k-1}^{\infty}$ is a compact-tight frame such that the frame operator is written in the following form $$S=\mathcal{K}+\xi I,$$ 
The method for creating compact-tight frames based on a given orthonormal basis is described in our following theorem.
\begin{theorem}\label{T4}
 Let $\left\{l_k\right\}_{k=1}^{\infty}$ be a sequence of real numbers greater than $\xi$ such that $\lim _{k \rightarrow \infty} l_k=\xi$ and $\xi$ be a positive real number. If $\left\{e_k\right\}_{k=1}^{\infty}$ is any orthonormal basis of $\mathcal{H}$, then $\left\{l_{k}^{\frac{1}{2}} e_k\right\}_{k=1}^{\infty}$ is a compact-tight frame for $\mathcal{H}$ with frame bounds $\xi$ and $\sup _k l_k$.
\end{theorem}
\begin{proof}
Let $f \in \mathcal{H}$, we have
$$
\sum_{k=1}^{\infty}\left\langle f, l_{k}^{\frac{1}{2}} e_k\right\rangle \left\langle  l_{k}^{\frac{1}{2}} e_k,f\right\rangle =\sum_{k=1}^{\infty} l_k\left\langle f, e_k\right\rangle \left\langle e_k,f \right\rangle
$$
implies that
$$
\xi\langle f,f\rangle\leq \sum_{k=1}^{\infty} \left\langle f, l_{k}^{\frac{1}{2}} e_k\right\rangle \left\langle  l_{k}^{\frac{1}{2}} e_k,f\right\rangle  \leq\left(\sup _k l_k\right)\langle f,f\rangle .
$$
Thus, $\left\{l_{k}^{\frac{1}{2}} e_k\right\}_{k=1}^{\infty}$ is a frame for $\mathcal{H}$ with frame bounds $\xi$ and $\sup _k l_k$. To show that this frame is compact-tight, note that
$$
\begin{aligned}
S f & =\sum_{k=1}^{\infty}\left\langle f, l_{k}^{\frac{1}{2}} e_k\right\rangle l_{k}^{\frac{1}{2}} e_k \\
& =\sum_{k=1}^{\infty} l_k\left\langle f, e_k\right\rangle e_k
\\
& =\sum_{k=1}^{\infty} (l_k+\xi-\xi)\left\langle f, e_k\right\rangle e_k \\
& =\sum_{k=1}^{\infty}\left(l_k-\xi\right)\left\langle f, e_k\right\rangle e_k+\xi \sum_{k=1}^{\infty}\left\langle f, e_k\right\rangle e_k \\
& =\mathcal{K} f +\xi f.
\end{aligned}
$$
Since $\lim _{k \rightarrow \infty} l_k=\xi$, Thus

$$
\mathcal{K}:=\sum_{k=1}^{\infty}\left(l_k-\xi\right)\left\langle., e_k\right\rangle e_k
$$
is a compact operator.
\end{proof}
\begin{example}
Let $\left\{e_k\right\}_{k=1}^{\infty}$ be an orthonormal basis for $\mathcal{H}$. Since
$$
l_k:=1+\exp\left(-\dfrac{k^2}{2}\right)>1
$$
for every $k$ and
$$
\lim _{k \rightarrow \infty} l_k=1,
$$
Theorem \ref{T4} proves the sequence
$
\left\{l_{k}^{\frac{1}{2}} e_k\right\}_{k=1}^{\infty}
$
is a compact-tight frame with bounds 1 and 2 .

\end{example}

 \begin{proposition}
 Let $\left\{g_k\right\}_{k=1}^{\infty}$ be a frame in $\mathcal{H}$ It is created from an orthonormal basis by repeatedly iterating a set number of times each basis member. Then, $\left\{g_k\right\}_{k=1}^{\infty}$ is a finite-rank-tight frame.
\end{proposition}  
 \begin{proof}
 Let $\left\{e_1, e_2, \ldots\right\}$ be an orthonormal basis for $\mathcal{H}$. Assume that $\left\{g_k\right\}_{k=1}^{\infty}$ is obtained by repeating the basis elements to create the basis. $e_{j_1}, \ldots, e_{j_n}$, such that $e_{j_{\mathrm{m}}}$ is repeated $\theta_m>1$ times for each $m \in\{1, \ldots, n\}$. Then,
$$
\begin{aligned}
S f & =\sum_{k=1}^{\infty}\left\langle f, g_k\right\rangle g_k \\
& =\sum_{m=1}^n\left(\theta_m-1\right)\left\langle f, e_{j_m}\right\rangle e_{j_m} +\sum_{i=1}^{\infty}\left\langle f, e_i\right\rangle e_i\\
& =\mathcal{K}f+ f .
\end{aligned}
$$
Then,
$
\mathcal{K}=\sum_{m=1}^n\left(\theta_m-1\right)\left\langle\cdot, e_{j_m}\right\rangle e_{j_m}
$
is a finite-rank operator.
 \end{proof}
 
\begin{theorem}
 The canonical dual of a compact-tight frame, is compact-tight.
\end{theorem}
\begin{proof}
 Let $\left\{f_k\right\}_{k=1}^{\infty}$ be a $(\xi, \mathcal{K})$-compact-tight frame such that the frame operator $S=\mathcal{K}+\xi I$.

In which case $\xi \neq 0$ and $\mathcal{K}$ is a compact operator. Then,
$$
S^{-1}=\mathcal{T}+\xi^{-1} I,
$$
where $\mathcal{T}$ is a compact operator. So all we have to do is choose $\mathcal{T}$ in such a way that
$$
\mathcal{T} S=-\xi^{-1} \mathcal{K} .
$$
So it is evident that $\mathcal{T}$ is a compact operator. Furthermore,
$$
\begin{aligned}
\left(\mathcal{T}+\xi^{-1} I\right) S & =\mathcal{T} S +\xi^{-1} S\\
& =\xi^{-1}(\mathcal{K}+\xi I)-\xi^{-1} \mathcal{K} \\
& =I .
\end{aligned}
$$
\end{proof}
\begin{theorem}\label{T49}
Let $\mathfrak{F}=\left\{f_k\right\}_{k=1}^{\infty}$ and $\mathfrak{G}=\left\{\tilde{f}_k\right\}_{k=1}^{\infty}$ be $\left(1, \mathcal{K}_1\right)$ - and $\left(1, \mathcal{K}_2\right)$ compact-tight frames, respectively. If there erists an infinite subset $\sigma$ of $\mathbb{N}$ such that $\left\{f_k\right\}_{k \in \sigma}$ and $\left\{\tilde{f}_k\right\}_{k \in \sigma^c}$ are orthonormal bases. Then, $\left\{f_k\right\}_{k=1}^{\infty}$ and $\left\{\tilde{f}_k\right\}_{k=1}^{\infty}$ cannot be woven.
\end{theorem}
\begin{proof}
Assume on the contrary that $\left\{f_k\right\}_{k=1}^{\infty}$ and $\left\{\tilde{f}_k\right\}_{k=1}^{\infty}$ are woven.
 Then, for the partition $\left\{\sigma, \sigma^c\right\}$, we obtain the frame $\left\{f_k\right\}_{k \in \sigma^c} \cup\left\{\tilde{f}_k\right\}_{k \in \sigma}$. According to the definition \ref{D42}, the frame operators for $\left\{f_k\right\}_{k-1}^{\infty}$ and $\left\{\tilde{f}_k\right\}_{k=1}^{\infty}$ can be written as
$$
\begin{aligned}
S_F f & =\sum_{k=1}^{\infty}\left\langle f, f_k\right\rangle f_k \\
& =\sum_{k \in \sigma}\left\langle f, f_k\right\rangle f_k+\sum_{k \in \sigma^c}\left\langle f, f_k\right\rangle f_k \\
& =\mathcal{K}_1 f+f
\end{aligned}
$$
and
$$
\begin{aligned}
S_G f & =\sum_{k=1}^{\infty}\left\langle f, \tilde{f}_k\right\rangle \tilde{f}_k \\
& =\sum_{k \in \sigma^c}\left\langle f, \tilde{f}_k\right\rangle \tilde{f}_k+\sum_{k \in \sigma}\left\langle f, \tilde{f}_k\right\rangle \tilde{f}_k \\
& =\mathcal{K}_2 f+f .
\end{aligned}
$$
Make clear the frame operator of $\left\{f_k\right\}_{k \in \sigma^c} \cup\left\{\tilde{f}_k\right\}_{k \in \sigma}$ by $S_{F+G}$. Then for $f \in \mathcal{H}$,
$$
\begin{aligned}
S_{F+G} f & =\sum_{k \in \sigma^c}\left\langle f, f_k\right\rangle f_k+\sum_{k \in \sigma}\left\langle f, \tilde{f}_k\right\rangle \tilde{f}_k \\
& =\mathcal{K}_1 f+\mathcal{K}_2 f + f \\
& =\left(\mathcal{K}_1+\mathcal{K}_2\right) f +f.
\end{aligned}
$$
Thus, $$S_{F+G}=\mathcal{K}_1+\mathcal{K}_2.$$ Since $\mathcal{K}_1$ and $\mathcal{K}_2$ are compact operators, $S_{F+G}$ is a compact operator, which is absurd.
\end{proof}
\medskip

\section*{Declarations}

\medskip

\noindent \textbf{Availability of Data and Materials}\newline
\noindent Not applicable.

\medskip

\noindent \textbf{Ethics Approval and Consent to Participate}\newline
\noindent It is important to note that this article does not involve any studies with animals or human participants.

\medskip

\noindent \textbf{Competing Interests}\newline
\noindent The authors have no conflicts of interest to declare.

\medskip

\noindent \textbf{Funding}\newline
\noindent There are no financial sources to declare for this paper.

\medskip

\noindent \textbf{Authors' Contributions}\newline
\noindent All authors have equally contributed to the conception and design of the study, drafting of the manuscript, sequence alignment, and have read and approved the final manuscript.


\medskip

\begin{thebibliography}{99}

\bibitem{Dau} Daubechies I., Grossmann A., Meyer Y., 1986. Painless nonorthogonal expansions, J. Math. Phys.
27, 1271–1283.

\bibitem{Duf} Duffin R. J, Schaeffer A. C, Trans. Amer. Math. Soc. 72,  341-366, 1952.
Doi: doi.org/10.1090/S0002-9947-1952-0047179-6

\bibitem{Ghiati} Ghiati, M., Rossafi, M., Mouniane, M. et al. Controlled continuous 
$\ast$-g-frames in Hilbert C*-modules. J. Pseudo-Differ. Oper. Appl. 15, 2 (2024). https://doi.org/10.1007/s11868-023-00571-1

\bibitem{Kap} Kaplansky I, \emph{Modules over operator algebras}, Amer. J. Math. 75, 839-858, 1953.
Doi: doi.org/10.2307/2372552

\bibitem{Kho2} Khorsavi A, Khorsavi B, \emph{Fusion frames and g-frames in Hilbert $\mathcal{C}^{\ast}$-modules}, Int. J. Wavelet, Multiresolution
and Information Processing 6, 433-446, 2008.
Doi: doi.org/10.1142/S0219691308002458

 \bibitem{Mass}  H. Massit, M. Rossafi, C. Park, Some relations between continuous generalized frames. Afr. Mat. 35, 12 (2024). https://doi.org/10.1007/s13370-023-01157-2

\bibitem{Pas}Paschke, W.: Inner product modules over B* -algebras. Trans. Am. Math. Soc. 182, 443–468 (1973)

\bibitem{r1} Rossafi M, Kabbaj S, $\ast$-K-operator Frame for $End_{\mathcal{A}}^{\ast}(\mathcal{H})$,  Asian-Eur. J. Math. 13 (2020).
Doi: doi.org/10.1142/S1793557120500606.

\bibitem{ROSS} Rossafi, M., Ghiati, M., Mouniane, M., Chouchene, F., Touri, A., Kabbaj, S.: Continuous frame in Hilbert C*-modules. J. Anal. (2023). https://doi.org/10.1007/s41478-023-00581-8

\bibitem{ROS} Rossafi, M., Nhari, F. D., Park, C., Kabbaj, S.: Continuous g-frames with C*-valued bounds and their properties. Complex Anal. Operator Theory 16, 44 (2022). https://doi.org/10.1007/s11785-022-01229-4

\bibitem{Sah} N. K. Sahu, Controlled g-frames in Hilbert C*-modules, Mathematical Analysis and its Contemporary Applications, 3 (2021), 65–82.

\bibitem{Zah} X. Zhao, P. Li, Weaving Frames in Hilbert C*- Modules, Hindawi
Journal of Mathematics
Volume 2021, Article ID 2228397, 13 pages, https://doi.org/10.1155/2021/2228397. 
 

 
	\end{thebibliography}
\end{document}